\documentclass{amsart}
\usepackage{amssymb}
\usepackage{hyperref}

\numberwithin{equation}{section}
\newtheorem{theorem}{Theorem}[section]
\newtheorem{conjecture}[theorem]{Conjecture}
\newtheorem{lemma}[theorem]{Lemma}
\newtheorem{remark}[theorem]{Remark}
\newtheorem{example}[theorem]{Example}
\newtheorem{examples}[theorem]{Examples}
\newtheorem{proposition}[theorem]{Proposition}
\newtheorem{corollary}[theorem]{Corollary}
\theoremstyle{definition}

 \DeclareMathOperator\Ad{{\rm Ad}}
 \DeclareMathOperator\Aut{{\rm Aut}}

 \DeclareMathOperator\Iso{{\rm Isom}}
 \DeclareMathOperator\Lie{{\rm Lie}}

 \DeclareMathOperator\ad{{\rm ad}}
 
 \DeclareMathOperator\rad{{\rm rad}}
 \DeclareMathOperator\rk{{\rm rk}}
 
 \DeclareMathOperator\tr{{\rm tr}}

\def\Sp{{\rm Sp}}

\newcommand\Ca{\mathbb O}

\newcommand\C{\mathbb C}

\newcommand\GL{\Gl}
\newcommand\Gl{{\rm GL}}

\newcommand\HH{{\rm H}}
\newcommand\LA{{\rm A}}

\newcommand\LE{{\rm E}}
\newcommand\LF{{\rm F}}
\newcommand\LG{{\rm G}}

\newcommand\R{\mathbb R}

\newcommand\SO{{\rm SO}}
\newcommand\SUxU[2]{ {\rm S(U(}#1) \times {\rm U(}#2{\rm ))} }
\newcommand\SU{{\rm SU}}
\newcommand\SL{{\Sl}}
\newcommand\Sl{{\rm SL}}
\newcommand\Spin{{\rm Spin}}
\newcommand\U{{\rm U}}
\newcommand\Z{\mathbb{Z}}
\newcommand\bl{\vspace{1em}\hfill\\}

\newcommand\eS{{\rm S}}
\newcommand\gl{\g\mathfrak l}
\newcommand\g{{\mathfrak g}}

\newcommand\h{{\mathfrak h}}

\newcommand\id{{\rm id}}

\newcommand\m{{\mathfrak m}}
\newcommand\p{{\mathfrak p}}

\newcommand\so{\mathfrak s \mathfrak o}
\newcommand\spin{\mathfrak s \mathfrak p \mathfrak i \mathfrak n}

\newcommand\stru{\rule[-.65em]{0em}{1.7em}}
\newcommand\triplestru{\rule[-1.95em]{0em}{4.4em}}
\newcommand\str{\rule[-.48em]{0em}{1.7em}}
\newcommand\suxu[2]{\mathfrak s(\mathfrak u({#1}) {+} \mathfrak u({#2}))}
\newcommand\su{\mathfrak s \mathfrak u}
\newcommand\s{\mathfrak s}

\newcommand\z{\mathfrak z}

\renewcommand\H{\mathbb{H}}
\renewcommand\O{{\rm O}}
\renewcommand\P{{\rm P}}

\renewcommand\k{{\mathfrak k}}
\renewcommand\l{\mathfrak l}

\renewcommand\sp{\mathfrak s \mathfrak p}

\begin{document}
\setcounter{tocdepth}{1}


\newcounter{refcounter}
\renewcommand\therefcounter{\bf (\arabic{refcounter})}
\newcommand\mylabel[1]{\refstepcounter{refcounter}\bf\therefcounter\label{#1}}

%
%
%

\pagenumbering{arabic}\pagestyle{plain}

\title{Duality of symmetric spaces and polar actions}

\author[A.~Kollross]{Andreas Kollross}

\address{Fachbereich Mathematik, Technische Universit\"{a}t Darmstadt,
Schlossgartenstr.~7, 64289 Darmstadt, Germany}

\email{kollross@mathematik.tu-darmstadt.de}


\date{\today}



\maketitle

\begin{abstract}
We study isometric actions on Riemannian symmetric spaces of noncompact type
which are induced by reductive algebraic subgroups of the isometry group. We
show that for such an action there exists a corresponding isometric action on a
dual compact symmetric space, which reflects many properties of the original
action. For example, the principal isotropy subgroups of both actions are
locally isomorphic and the dual action is (hyper)polar if and only if the
original action is (hyper)polar. This fact provides many new examples for polar
actions on symmetric spaces of noncompact type and we use duality as a method
to study polar actions by reductive algebraic subgroups in the isometry group
of an irreducible symmetric space. Among other applications, we show that they
are hyperpolar if the space is of type III and of higher rank; we prove that
such actions are orbit equivalent to Hermann actions if they are hyperpolar and
of cohomogeneity greater than one. Furthermore, we classify polar actions by
reductive algebraic subgroups of the isometry group on noncompact symmetric
spaces of rank one.
\end{abstract}


\section{Introduction}
\label{Intro}


A proper isometric action of a Lie group on a Riemannian manifold is called
{\em polar} if there is a complete immersed submanifold which intersects the
orbits perpendicularly and meets all orbits. Such a submanifold is called a
{\em section} for the action. A special case, which occurs in many natural
examples, is when the section is flat in its induced Riemannian metric. In this
case, the action is called {\em hyperpolar}.

Sections of polar actions are always totally geodesic submanifolds. If one
intends to study polar actions of cohomogeneity greater than one, it is
therefore natural to consider the class of Riemannian symmetric spaces, which
-- unlike generic Riemannian manifolds -- admit many nontrivial totally
geodesic submanifolds of dimension greater than one. In fact, the theory of
symmetric spaces is the main source of examples for polar actions. For
instance, the action of an isotropy group of a Riemannian symmetric space is
always hyperpolar. Polar and hyperpolar actions have been studied by many
authors, see e.g.\ the survey \cite{thorbergssonTGSG} for the history of the
subject and a bibliography. Among the hyperpolar actions, the most prominent
special case is the case of {\em cohomogeneity one actions}, i.e.\ actions
where there are orbits of codimension one. Cohomogeneity one actions on
spheres, complex and quaternionic projective space and on the Cayley plane have
been classified in~\cite{hsl}, \cite{takagi}, \cite{datri}, \cite{iwata}.
Hyperpolar actions on compact irreducible Riemannian symmetric spaces have been
classified by the author in~\cite{hyperpolar}. See \cite{BDRT08}, \cite{bt1},
\cite{bt2}, \cite{bt3} for classification results concerning hyperpolar actions
on noncompact symmetric spaces.

If one studies the more general case of polar actions on compact symmetric
spaces, where one does not require the sections to be flat, there is a sharp
contrast between the case of rank one symmetric spaces on the one hand and of
irreducible spaces of higher rank on the other hand. Namely, while there are
many examples of such actions on the spaces of rank one, see~\cite{dadok} and
\cite{pth1} for a classification, there is not even one nontrivial example
known on the irreducible spaces of higher rank and it is an interesting problem
to decide if there are any such actions at all.

The first result in this connection was proved in \cite{brueck}, where it was
shown that a polar action with a fixed point on an irreducible symmetric space
of higher rank is hyperpolar. It has been shown by the author in~\cite{polar}
that polar actions are hyperpolar on the symmetric spaces with simple compact
isometry group and rank greater than one. Earlier, Biliotti~\cite{biliotti} had
completed the classification of all coisotropic and polar actions on compact
irreducible Hermitian symmetric spaces, see also \cite{bg} and \cite{pth2}. His
result lead him to make the following conjecture.
\begin{conjecture}\textup{\cite{biliotti}}\label{ConjBiliotti}
A nontrivial polar action of a connected Lie group on an irreducible compact
Riemannian symmetric space of rank greater than one is hyperpolar.
\end{conjecture}
In \cite{polex}, the author has shown that the conjecture also holds for the
symmetric spaces given by compact simple Lie groups of exceptional type endowed
with a biinvariant Riemannian metric. This is still an open problem for compact
Lie groups of classical type. Note that we cannot drop the irreducibility
assumption in the conjecture since otherwise e.g.\ products of transitive with
trivial actions would be counterexamples.

It is an intriguing problem to classify polar and hyperpolar actions in more
general settings, in particular, to ask if a statement similar to
Conjecture~\ref{ConjBiliotti} holds for actions on symmetric spaces of the
noncompact type. However, the straightforward generalization of
Conjecture~\ref{ConjBiliotti}, where one simply drops the hypothesis that the
space acted upon be compact, is known not to be true. Indeed, in
\cite[Proposition 4.2]{BDRT08}, examples of homogeneous foliations (i.e.\
actions where all orbits are principal) are given which are polar, but not
hyperpolar. Note that there are no homogeneous polar foliations on irreducible
compact symmetric spaces, cf.\ \cite[Lemma 1A.2]{pth1}. On the other hand, it
follows from Cartan's fixed point theorem and the result of Br\"{u}ck~\cite{brueck}
that the conjecture still holds for actions on noncompact symmetric spaces if
one requires the group which acts polarly to be compact. Thus it is an
interesting question to what extent properties of polar actions generalize from
the compact to the noncompact setting.

However, while there are a number of strong results, e.g.\ \cite{hyperpolar},
\cite{polar}, \cite{polex}, \cite{pth1} for polar and hyperpolar actions in
the realm of compact symmetric spaces, classification results on the noncompact
side are only available in special cases and for a limited class of spaces see
e,g.\ \cite{BDRT08}, \cite{bt1}, \cite{bt2}, \cite{bt3}, \cite{brueck},
\cite{dk}, \cite{wu}. At first glance, this seems remarkable in view of the
fact that there is a well known duality between Riemannian symmetric spaces of
the compact and the noncompact type. In fact, for every symmetric space of the
noncompact type $M$ there is a dual compact symmetric space $M^*$ which closely
reflects some geometric aspects of its noncompact counterpart, and vice versa.
Moreover, there are many examples of group actions on symmetric spaces for
which there exist obvious analogues on the dual space.

But the connection established by duality is far from being a complete
one-to-one correspondence if one considers isometric Lie group actions. Namely,
there is an obvious bijective map
\begin{equation}\label{EqnDualityMap}
X + Y \ \longmapsto \ X + iY \quad \hbox{for}\quad X \in \k,\;\; Y \in \p
\end{equation}
between the Lie algebras of the isometry groups of~$M$ and $M^*$, but this map
is not a Lie algebra homomorphism and does not, in general, map subalgebras
onto subalgebras. In fact, there are some phenomena like horocycle foliations
on noncompact spaces which appear not to have a counterpart on a dual compact
space. Thus the methods used in the compact case cannot be applied to
noncompact spaces in general.

Then again, there is a special case where duality can be applied, namely when
the Lie algebra $\h$ of the group $H \subseteq G$ acting on the symmetric space
$M = G/K$ of noncompact type is invariant under the Cartan involution. Then the
image of  $\h$ under the map (\ref{EqnDualityMap}) is a subalgebra $\h^*
\subseteq \g^*$ and this defines an action on a compact symmetric space $M^*$
dual to $M$.

It is the content of the Karpelevich-Mostow Theorem, see Section~\ref{KM}, that
a semisimple subalgebra $\h$ in a semisimple Lie algebra $\g$ none of whose
ideals is compact is always conjugate to a subalgebra invariant under a Cartan
decomposition, or equivalently, $H$ always has a geodesic orbit. More
generally, the same conclusion holds if $H \subseteq G$ is a reductive
algebraic subgroup. Hence a dual action exists for such actions.  As we will
show, such a dual action on $M^*$ has many properties in common with the
original action on~$M$, in particular, the action on $M^*$ is (hyper)polar if
and only if the action on $M$ is (hyper)polar. We will use this fact to obtain
a number of new results on polar and hyperpolar actions on noncompact symmetric
spaces by applying duality to earlier results in the compact setting. Our
method is a generalization of~\cite{dk}, where dual actions are considered in
the special case of actions with a fixed point. In a similar fashion, duality
was used in \cite{bt1} to study cohomogeneity one actions on noncompact
symmetric spaces.

As the correspondence between $M$ and $M^*$ is defined by a map on the Lie
algebra level, the construction of dual actions depends on the choice of the
reference point. For the action of a reductive algebraic subgroup $H \subseteq
G$ on a noncompact symmetric space $M=G/K$, the type of the totally geodesic
orbit is unique and therefore the dual action of~$H^*$ on $M^* = G^* / K^*$ is
also unique up to coverings of~$M^*$ (we do not assume $M^*$ to be simply
connected) and conjugacy of~$H^*$. On the other hand, isometric actions on
compact symmetric spaces may have various types of totally geodesic orbits or
no totally geodesic orbits at all. Thus the map $(H) \mapsto (H^*)$ which maps
the conjugacy classes of reductive algebraic subgroups $H \subseteq G$ to
conjugacy classes of subgroups $H^* \subseteq G^*$ is in general neither
injective nor surjective. This phenomenon will be illustrated by several
examples.


I would like to thank Jos\'{e} Carlos D\'{\i}az-Ramos and Antonio J.\ Di Scala for
discussions.


\section{Preliminaries}
\label{Prelim}


Let $\g$ be a real semisimple Lie algebra and let $B(\cdot,\cdot)$ be its
Killing form. An {\em involution} on $\g$ is a Lie algebra automorphism
$\theta$ of~$\g$ such that $\theta^2 = \id_{\g}$.  (Note that in our definition
the involution $\theta$ may be trivial, i.e.\ we may have $\theta = \id_{\g}$.)
Such an involution is called a {\em Cartan involution} on $\g$ if
$B_\theta(X,Y) = -B(X,\theta Y)$ is a positive definite bilinear form. Any real
semisimple Lie algebra has a Cartan involution and any two Cartan involutions
are conjugate by an inner automorphism of~$\g$. For a Cartan involution $\theta
\colon \g \to \g$, we call
\begin{equation}\label{EqnCartanDec}
\g = \k \oplus \p,
\end{equation}
where $\k$ is the $(+1)$-eigenspace and $\p$ is the $(-1)$-eigenspace
of~$\theta$, the {\em Cartan decomposition} corresponding to~$\theta$. Note
that $\k$ is a maximal compact subalgebra of~$\g$. It follows that $\g^* := \k
\oplus i \p \subset \g(\C) = \g \otimes \C$ is a compact real semisimple Lie
algebra, where $i = \sqrt{-1}$ is the imaginary unit. We say that $\g^*$ is the
Lie algebra {\em dual} to $\g$ with respect to the involution~$\theta$.
Moreover, we say that a subalgebra $\h \subseteq \g$ is {\em canonically
embedded} with respect to the Cartan decomposition (\ref{EqnCartanDec}) if
$\theta(\h) = \h$ or, equivalently,
\begin{equation}\label{EqnCartanDec2}
\h = (\h \cap \k) \oplus (\h \cap \p).
\end{equation}
If $\h \subseteq \g$ is canonically embedded then $\h^* := (\h \cap \k) \oplus
i(\h \cap \p)$ is a subalgebra of~$\g^*$.

If the pair of Lie groups $(G,K)$ where $G$ is a semisimple Lie group and $K$ a
compact subgroup corresponds to the pair of Lie algebras $(\g,\k)$ and the pair
of compact Lie groups $(G^*,K^* )$ corresponds to $(\g^*,\k)$, we say that the
compact symmetric space $M^* = G^*/K^*$ equipped with a Riemannian metric
induced by the negative of the Killing form on~$\g^*$ is a {\em compact dual}
of the symmetric space $M = G/K$.

In a wider sense, we call two symmetric spaces $X$ and $Y$ {\em dual} to each
other if there exist decompositions of the universal coverings $\tilde X = X_1
\times \ldots \times X_n$ and $\tilde Y = Y_1 \times \ldots \times Y_n$ such
that for each $j \in \{1,\ldots,n\}$, either $X_j$ and $Y_j$ are Euclidean of
the same dimension or $X_j$ and $Y_j$ are irreducible symmetric spaces and such
that one is the compact dual of the other.

Let the Riemannian metric on $M = G/K$ be given by a scalar product $\beta
\colon \p \times \p \to \R$. We will henceforth assume that a compact dual of a
symmetric space $M = G/K$ is equipped with the corresponding {\em dual metric}
given by the scalar product $\beta^* \colon i\p \times i\p \to \R$, which we
define by $\beta^*(X,Y) := \beta(iX,iY)$. In particular, if $M$ is endowed with
the Riemannian metric induced by the Killing form of~$\g$, then the dual metric
is the Riemannian metric induced by the negative of the Killing form of~$\g^*$.
However, in the applications we study in Sections~\ref{Applications},
\ref{PAHS}, \ref{PACH}, \ref{PAQH} and \ref{PACayHP}, the symmetric spaces
under consideration are irreducible and any $G$-invariant Riemannian metric is
unique up to a constant scaling factor (whose choice is irrelevant here).

Consider the action of a group~$H$ on a set~$M$, denoted by $H \times M \to M$,
$(h,m) \mapsto h \cdot m$. For a point $p \in M$, we define $H_p := \{ h \in H
\mid h \cdot p = p \} \subseteq H$ to be the {\em stabilizer} or {\em isotropy
subgroup} at the point~$p$ and by $H \cdot p := \{ h \cdot p \mid h \in H\}$ we
denote the {\em orbit} of the $H$-action through the point~$p$. Since the
isotropy subgroups of the points along an orbit are conjugate, we may define
the {\em orbit type} of the orbit $H \cdot p$ as the conjugacy class $(H_p)$ of
the subgroup $H_p$ in $H$. This defines a partial order on the set of orbits:
We define $H \cdot p \preccurlyeq H \cdot q$ if and only if there is an element
$h \in H$ such that $h H_p h^{-1} \supseteq H_q$.

An action of a Lie group $H$ on a manifold $M$ is called {\em proper} if the
map $G \times M \to M \times M$, $(g, p) \mapsto (g \cdot p, p)$ is proper. A
proper isometric action of a Lie group $H$ on a Riemannian manifold is called
{\em polar} if there exists a complete immersed submanifold $\Sigma$ which
meets all the orbits of the group action, i.e.\ $G \cdot \Sigma = M$, and in
such a way that each intersection between $\Sigma$ and an orbit is orthogonal,
i.e.\ $T_p\Sigma \perp T_p(H \cdot p)$ for all $p \in \Sigma$. Such a
submanifold $\Sigma$ is called a {\em section} for the $H$-action. It is well
known that sections are totally geodesic submanifolds. All sections of a polar
actions are conjugate by the group action. Obvious examples of polar actions
are given by transitive actions, where the points of the manifold are sections,
and also by actions with discrete orbits, where the manifold itself is a
section. We will tacitly assume polar actions to be nontrivial in the sense
that the orbits are of positive dimension since otherwise one gets technical
counterexamples e.g.\ for Corollary~\ref{CorSection} and
Theorems~\ref{ThPolHyp}, \ref{ThExcPolHyp} below.

By $\Iso(M)$ we will denote the group of isometries of a Riemannian manifold,
by $\Iso(M)_0$ its connected component. If a Lie group $H$ acts isometrically
and effectively on a Riemannian manifold, we may assume that $H \subseteq
\Iso(M)$. We say that two Lie group actions on two Riemannian manifolds $M_1$
and $M_2$ are {\em conjugate} if there is an equivariant isometry $M_1 \to
M_2$. The actions of two subgroups $H_1, H_2 \subseteq \Iso(M)$ on a Riemannian
homogeneous space~$M$ are conjugate if and only if $H_1$ and $H_2$ are
conjugate in $\Iso(M)$.

For proper Lie group actions on connected manifolds we have the following well
known facts, see \cite{palais}. There is a uniquely determined maximal orbit
type of the $H$-action. The orbits which are of this type are called {\em
principal orbits} of the $H$-action, the corresponding isotropy subgroups are
called {\em principal isotropy subgroups}. The union of principal orbits is an
open and dense subset of~$M$. The codimension of a principal orbit in $M$ is
called the {\em cohomogeneity} of the action. At any point $p \in M$, the
isotropy subgroup $H_p$ acts (by the differentials at~$p$ of the maps $x
\mapsto g \cdot x$) on the tangent space $T_p M$. For this linear action, the
tangent space $T_p(H \cdot p)$ and the normal space $N_p(H \cdot p)$ are
invariant subspaces; the action of~$H_p$ on $N_p(H \cdot p)$ thus defined is
called the {\em slice representation} of the $H$-action at the point~$p$. The
slice representation is trivial if and only if the orbit through $p$ is
principal. The Slice Theorem asserts that a tubular neighborhood of an orbit $H
\cdot p$ is equivariantly diffeomorphic to a tubular neighborhood around the
zero section in the normal bundle $H \times_{H_p} N_p(H \cdot p)$, where $H_p$
acts by the slice representation on the normal space $N_p(H \cdot p)$. In
particular, the principal isotropy subgroup of the $H$-action on $M$ is
conjugate to any principal isotropy subgroup of an arbitrary slice
representation and thus the cohomogeneity of each slice representation is the
same as the cohomogeneity of the $H$-action on~$M$. Slice representations of
polar actions are polar~\cite [Theorem~4.6] {pt}. For a polar action, the
dimension of a section equals the cohomogeneity of the action.

Let $M$ be a Riemannian symmetric space and let $p \in M$. Let $G = \Iso(M)_0$
and let $K = G_p$. An action of a closed subgroup $H \subset G$ is called {\em
Hermann action} if there is an involutive automorphism $\sigma \colon \g \to
\g$ such that $\h = \g^\sigma$, where $\g^\sigma$ denotes the fixed point set
of~$\sigma$. It was shown by Hermann~\cite{hermann} that these actions are
hyperpolar on compact symmetric spaces. We say that two isometric actions on
two Riemannian manifolds $M$ and $N$ are {\em orbit equivalent} if there is an
isometry $F \colon M \to N$ which maps each connected component of an orbit
in~$M$ onto a connected component of an orbit in~$N$. Obviously, the
(hyper-)polarity of an action depends only on its orbit equivalence class.


\section{The Karpelevich-Mostow Theorem}
\label{KM}


The following Theorems~\ref{ThKarp} and \ref{ThMost} are equivalent and their
content is called the {\em Karpelevich-Mostow Theorem}. Its geometric version
was proved by Karpelevich~\cite{karpelevich}.
\begin{theorem}\label{ThKarp}
Let $M$ be a symmetric space of non-positive curvature without flat factors.
Then any connected and semisimple subgroup $H \subseteq \Iso(M)$ has a totally
geodesic orbit $H \cdot p \subseteq M$.
\end{theorem}
The algebraic version can be stated as follows.
\begin{theorem}\label{ThMost}
Let $\g$ be a real semisimple Lie algebra such that each simple ideal is
noncompact and let $\h \subseteq \g$ be a semisimple subalgebra. Then $\h$ is
canonically embedded with respect to some Cartan decomposition of~$\g$.
\end{theorem}
In this form, the statement was proven by Mostow~\cite[Theorem 6]{mostow}.
Recently, a geometric proof was obtained by Di Scala and Olmos~\cite{dso}.
There exists a generalization of the Karpelevich-Mostow Theorem~\ref{ThKarp}
for the actions of reductive algebraic subgroups of $\Iso(M)$ on $M$, see
\cite{OV} for details. Let us briefly review the definitions necessary to
formulate this more general statement. Let $\g$ be a semisimple complex Lie
algebra. One may identify $\g$ with the linear complex Lie algebra $\ad\g
\subseteq \gl(\g)$ and thus one can define the notion of an algebraic
subalgebra of~$\g$. A subalgebra $\h \subseteq \g$ is called an {\em algebraic
subalgebra} of~$\g$ if $\h \subseteq \g$ is the Lie algebra of some algebraic
subgroup of the complex algebraic group $\GL(\g)$. Furthermore, a subalgebra
$\h \subseteq \g$ is called {\em reductive subalgebra} if the radical of~$\h$
consists of elements which are semisimple in $\g$, i.e.\ the maps $\ad z$ are
semisimple linear endomorphisms of~$\g$ for all $z \in \rad(\h)$. Equivalently,
$\h \subseteq \g$ is called a {\em reductive subalgebra} if it can be written
as $\h = \z(\h) \oplus \h'$ where the center $\z(\h)$ consists of semisimple
elements of~$\g$ and where the derived subalgebra $\h'$ is semisimple. If an
algebraic subalgebra of~$\g$ is reductive in the sense just defined, we call it
a {\em reductive algebraic subalgebra} of~$\g$. For a real semisimple Lie
algebra $\g$ we say that a subalgebra $\h \subseteq \g$ is {\em (reductive)
algebraic} if its complexification $\h(\C) = \h \otimes \C$ is a (reductive)
algebraic subalgebra of $\g(\C) = \g \otimes \C$.

\begin{theorem}\label{ThOniVin}\cite[Theorem 3.6, Ch.~6]{OV}\hfill\\
An algebraic subalgebra of a real semisimple Lie algebra~$\g$ is reductive if
and only if it is canonically embedded in $\g$ with respect to some Cartan
decomposition of~$\g$.
\end{theorem}

Note that in particular any semisimple subalgebra of a real (or complex)
semisimple Lie algebra is a reductive algebraic subalgebra. The theorem holds
also for a compact Lie algebra~$\g$, but since a Cartan decomposition is
trivial for compact $\g$, the assertion of the theorem is void in this case.

Let $G$ be semisimple real Lie group. We say that a subgroup $H \subseteq G$ is
a {\em reductive algebraic subgroup} if the Lie algebra $\h$ of~$H$ is a
reductive algebraic Lie algebra of~$\g$. One has to be careful not to confuse
the two notions of a reductive {\em sub}algebra of a Lie algebra on the one
hand and of a reductive Lie algebra on the other hand. (A Lie algebra is said
to be {\em reductive} if it is a direct sum of an abelian and a semisimple Lie
algebra.) Indeed, each non-semisimple element of a Lie algebra spans a
one-dimensional, hence abelian subalgebra which is not a reductive subalgebra,
cf.\ Example~\ref{ExplNilp}.

\begin{remark}\label{RemKM}\rm
Let $M$ be a symmetric space of non-positive curvature without flat factors.
Let $G = \Iso(M)_0$ be the connected component of the isometry group of~$M$.
Then $G$ is semisimple. Assume that $H \subseteq G$ is a connected reductive
algebraic subgroup. Then there is a point $q$ such that $H \cdot q$ is a
totally geodesic submanifold of~$M$.

In fact, it follows from Theorem~\ref{ThOniVin} that there is a point $q \in M$
such that $\h$ is canonically embedded, i.e.\ (\ref{EqnCartanDec2}) holds, with
respect to the Cartan decomposition $\g = \k \oplus \p$ where $K = G_q$ is the
stabilizer of~$q$ in~$G$. In this case, $\h \cap \p \subseteq \g$ is a Lie
triple system and it follows that the $H$-orbit through~$q$ is a totally
geodesic submanifold of~$M$.
\end{remark}

\begin{proposition}\label{PropKM}
Let $M$ be a simply connected symmetric space of non-positive curvature. Let
$H$ be a connected subgroup of the isometry group of~$M$. Assume there is a
point $q \in M$ such that the orbit $H \cdot q$ is a totally geodesic
submanifold of~$M$. Then the following statements are true.
\begin{enumerate}

\item As a differentiable $H$-manifold, $M$ is equivariantly diffeomorphic
    to the normal bundle of the orbit $H \cdot q$.

\item The orbit type $(H_q)$ of $H \cdot q$ is minimal, i.e.\ we have
    $(H_p) \succcurlyeq (H_q)$ for all $p \in M$.

\item The isotropy subgroup $H_q \subseteq H$ is a maximal compact
    subgroup.

\end{enumerate}
\end{proposition}

\begin{proof}
To prove part~(i), we first show that the totally geodesic orbit $H \cdot q$ is
totally convex, i.e.\ each geodesic segment $\gamma$ is contained in $H \cdot
q$ whenever the endpoints of~$\gamma$ are contained in $H \cdot q$. Using the
isometric $H$-action, we may restrict ourselves to consider geodesic segments
starting at~$q$. The Riemannian exponential $\exp \colon T_qM \to M$ is a
diffeomorphism~\cite[Ch.~VI, Theorem~1.1(iii)]{helgason} and $H \cdot q$ is the
image of the linear subspace $T_q(H \cdot q) \subseteq T_qM$ under this
diffeomorphism. Thus any geodesic segment starting at~$q$ has its endpoint in
$H \cdot q$ if and only if the segment is completely contained in $H \cdot q$.
Moreover, it follows that $H \cdot q$ is a closed subset of~$M$. By
\cite[Lemma~3.1]{bon}, a submanifold $V$ of a complete Riemannian manifold~$M$
of non-positive curvature is closed and totally convex if and only if $V$ is
totally geodesic and the Riemannian exponential map $\exp \colon NV \to M$ is a
diffeomorphism. To show the equivariance property, it suffices to note that for
any normal vector $v \in N(H \cdot q)$, the geodesic segment parametrized by
$\exp( t v )$, $t \in [0,1]$, is the unique shortest geodesic segment joining
$\exp(v)$ and $H \cdot q$. Since $H$ acts by isometries, we have $h \cdot
\exp(v) = \exp( h \cdot v )$.

Part(ii) follows immediately from part~(i). Indeed, for each $p \in M$ there is
a unique $v \in N(H \cdot q)$ such that $\exp(v)=p$ and it follows that $H_p
\subseteq H_x$ where $x \in H \cdot q$ is the unique point such that $v \in
N_x(H \cdot q)$.

Let $Q \subseteq H$ be a compact subgroup. By Cartan's fixed point theorem, the
$H$-action on $M$ restricted to~$Q$ has a fixed point~$p \in M$. Then $Q
\subseteq H_p$ and it follows from~(ii) that $Q$ is conjugate to a subgroup of
$H_q$. This proves~(iii).
\end{proof}

From Proposition~\ref{PropKM}~(ii) it follows that all totally geodesic orbits
of~$H$ are of the same (minimal) orbit type.

\begin{examples}\label{RemCpt}\rm
We remark that a statement analogous to Theorem~\ref{ThOniVin} does not hold
for symmetric spaces of compact type.  In fact, there are many examples of
nontrivial actions of compact groups on compact symmetric spaces which do not
have any totally geodesic orbits at all. Note that a closed subgroup of a
compact Lie group is always a reductive algebraic subgroup.

{(i)} Consider a Hermann action of a closed subgroup~$H \subset G$ on a compact
irreducible symmetric space $G/K$, where the isometry group $G$ is simple. Let
$\sigma, \theta \colon \g \to \g$ be involutive automorphisms such that $\k =
\g^\sigma$, $\h= \g^\theta$ are the fixed point sets of~$\sigma$ and $\theta$,
respectively. It was shown in~\cite{hertgo} that the $H$-action on $G/K$ has a
totally geodesic orbit if and only if there is an element $g \in G$ such that
$\Ad(g) \circ \sigma \circ \Ad(g)^{-1}$ commutes with $\theta$.
Conlon~\cite{conlon} determined all pairs of involutions on simple compact Lie
groups where no such element~$g$ exists. For example, the action of $H=\Sp(n)$
on $G/K=\SU(2n)/\SUxU{2n-1}1$ does not have any totally geodesic orbit for $n
\ge 2$.

{(ii)} For a different type of example, note that actions on the sphere~$\eS^n$
which are induced by irreducible orthogonal representations on $\R^{n+1}$ do
not have any totally geodesic orbit, except $\eS^n$ itself in case the action
is transitive. Indeed, the totally geodesic submanifolds of~$\eS^n$ are
precisely the intersections of~$\eS^n$ with linear subspaces of $\R^{n+1}$ and
hence a totally geodesic orbit spans an invariant subspace of $\R^{n+1}$.

{(iii)} Even if there are totally geodesic orbits in the compact setting, there
might not be a decomposition as in (\ref{EqnCartanDec2}). For example, let
$H=\SU(m)$ act on $\R^{2m+1}$ such that a linear action of~$H$ is given by the
standard representation of~$H$ on $\R^{2m}=\C^m$ plus a one-dimensional trivial
module. Then $H$ acts on $\eS^{2m}$ in such a fashion that there is a totally
geodesic orbit $H \cdot q$ and $H_q \cong \SU(m-1)$. But $(\SU(m),\SU(m-1))$ is
not a symmetric pair~\cite{helgason}. This example also shows that in the
compact setting, the orbit type of totally geodesic orbits may not be unique:
Apart from the one principal totally geodesic orbit, there are also two fixed
points.
\end{examples}

\begin{example}\label{ExplNilp}\rm
Let us give a simple example of an action on the hyperbolic plane, where the
group which acts is an algebraic, but not reductive algebraic, subgroup of the
isometry group and where there is no totally geodesic orbit. Let $\HH^2 =
\left\{ z \in \C \mid \Im(z) > 0 \right\}$ be the upper half-plane endowed with
the Riemannian metric $\smash{\Im(z)}^{-2}dzd\bar z$. Consider the isometric
action of $\Sl(2,\R)$ given by the transformations
$$
\begin{pmatrix}
  a & b \\
  c & d \\
\end{pmatrix} \cdot z = \frac{az+b}{cz+d}.
$$
Consider the subgroup $H \subset \SL(2,\R)$ consisting of all matrices where
$a=d=1$ and $c=0$. This group is isomorphic to the additive group of~$\R$ and
it acts on $\HH^2$ by horizontal translations, hence the $H$-orbits are the
horospheres given by the horizontal lines $\Im(z) = \hbox{const}$.  None of
these orbits is totally geodesic, as the geodesics in $\HH^2$ are orthogonal
arcs to the real axis or straight vertical half-lines ending on the real axis.
The group $H$ is obviously an algebraic subgroup of~$\Sl(2,\R)$. Its Lie
algebra~$\h$ is not a reductive subalgebra of ${\mathfrak s \mathfrak
l}(2,\R)$, since $\ad z \colon \g(\C) \to \g(\C)$ is not semisimple for $z \in
\h(\C) \setminus \{0\}$. Note that this action is of cohomogeneity one, hence
hyperpolar.
\end{example}

There is the following criterion for an algebraic subalgebra of semisimple Lie
complex Lie algebra to be reductive.

\begin{proposition}
Let $\h \subseteq \g$ be an algebraic subalgebra of the semisimple complex Lie
algebra~$\g$. Then $\h$ is a reductive algebraic subalgebra if and only if the
restriction of the Killing form $B(x,y) := \tr (\ad x \circ \ad y)$ to $\h
\times \h$ is non-degenerate.
\end{proposition}

\begin{proof}
See \cite[Ch.~4, Theorem~2]{OV}.
\end{proof}

\begin{example}\label{ExplEuclHyp}\rm
The following is a generalization of Example~\ref{ExplNilp}. Consider the upper
half space $\smash\HH^n = \{ (x_1, \ldots, x_n) \in \R^n \mid x_n > 0 \}$
endowed with the Riemannian metric $x_n^{-2}(dx_1^2 + \ldots + dx_n^2)$. Let $U
\subseteq \R^{n-1}$ be a linear subspace and let $p \in \HH^n$. Then the
additive group $U$ acts effectively on $\HH^n$ such that the orbit through a
point $p \in \HH^n$ is given by $p + (U \times \{0\})$. This action has no
totally geodesic orbit and the subgroup of $\Iso(\HH^n)$
given by the $U$-action is not reductive algebraic unless $U = \{0\}$. Let
$U^\perp \subseteq \R^{n-1}$ be the orthogonal complement of~$U$ in~$\R^{n-1}$
with respect to the standard scalar product on $\R^{n-1}$. Then the subspace $
\Sigma := \{ (v, y) \in \R^n \mid v \in U^\perp,\, y > 0 \} $ is a section for
the $U$-action on $\HH^n$ and we see that the $U$-action on $\HH^n$ is polar.
This is an example of a polar homogeneous foliation on $\HH^n$, since all
points of $\HH^n$ lie in a principal orbit of the $U$-action.
\end{example}

\begin{example}\label{ExplEHP}\rm
Using the same notation as in Example~\ref{ExplEuclHyp}, let $\varrho \colon L
\to \O(U^\perp)$ be a polar representation of the Lie group~$L$ and let the
linear subspace $\Sigma_0 \subseteq U^\perp$ be a section. Then $\Sigma := \{
(v, y) \in \R^n \mid v \in \Sigma_0,\, y > 0 \}$ is a section for the action of
$U \times L$ on $\HH^n$ given by $(u,\ell) \cdot (v+w,y) := (v+u +
\varrho(\ell)w, y)$ for $(u,\ell) \in U \times L$, $v \in U$, $w \in U^\perp$,
$y > 0$. In the special case where $\varrho$ is a trivial representation, this
is Example~\ref{ExplEuclHyp}. Obviously, this action has no totally geodesic
orbits, unless $U=\{0\}$.
\end{example}


\section{Dual actions}
\label{dual}


Let $\g$ be a semi-simple real Lie algebra all of whose simple ideals are
noncompact and let $\g = \k \oplus \p$ be a Cartan decomposition. Then the Lie
algebra~$\g^*$, defined by $\g^* := \k \oplus i \p \subset \g(\C)$ is a compact
real form of $\g(\C)$. We may define a map $\psi \colon \g \to \g^*$ by $X + Y
\mapsto X + i Y$ for $X \in \k$, $Y \in \p$ as in~(\ref{EqnDualityMap}).
Obviously, $\psi$ is a bijective $\R$-linear map, but not a homomorphism of Lie
algebras. If $\h \subseteq \g$ is a reductive algebraic subalgebra, it is
possible to apply the duality of symmetric spaces to the $H$-action on~$M$.
First note that we may assume, by replacing $H$ with a suitable conjugate
subgroup, that the point $q$ as given in Remark~\ref{RemKM} agrees with $[e] =
eK$. It follows that
\begin{equation}\label{EqnDualDef}
\h^* := \psi(\h) \subseteq \g^*
\end{equation}
is a subalgebra and $\h \cap i\p \subseteq \g^*$ is a Lie triple system. Now
let $G^*$ be some compact Lie group with Lie algebra $\g^*$ and let $K^*$ be
the connected subgroup of~$G^*$ corresponding to the subalgebra $\k \subseteq
\g^*$. Let $H^*$ be the connected subgroup of~$G^*$ corresponding to $\h^*$.
Then we say that the $H^*$-action on $G^*/K^*$ is {\em dual} to the $H$-action
on~$G/K$. It follows that $H^* \subseteq G^*$ is a reductive algebraic subgroup
of~$G^*$ and hence compact.

\begin{theorem}\label{ThDualAct}
Let $M$ be a symmetric space of non-positive curvature without flat factors.
Let $H$ be a connected reductive algebraic subgroup of the isometry group
of~$M$. Let $M^*$ be a compact symmetric space dual to~$M$ and let $H^*$ be a
subgroup of the isometry group of~$M^*$ such that the $H^*$-action on $M^*$ is
dual to the $H$-action on $M$. Then there exist points $q \in M$ and $q^* \in
M^*$ such that the following are true.
\begin{enumerate}

\item The $H$-orbit through $q$ is of minimal orbit type.

\item The orbits $H \cdot q \subseteq M$ and $H^* \cdot q^* \subseteq M^*$
    are totally geodesic.

\item The symmetric space $H^* \cdot q^*$ is dual to the symmetric space $H
    \cdot q$

\item The isotropy subgroups $H_q \subseteq H$ and $H^*_{q^*} \subseteq
    H^*$ are locally isomorphic.

\item The slice representations of~$H_q$ and $H^*_{q^*}$ are equivalent on
    the Lie algebra level. In particular, the $H$-action on $M$ and the
    $H^*$-action on~$M^*$ have the same cohomogeneity.

\end{enumerate}
\end{theorem}

\begin{proof}
Assume the orbit $H \cdot q$ is as described in Remark~\ref{RemKM}. Part~(i)
was shown in Proposition~\ref{PropKM}. We may assume without limitation of
generality that $q = [e]$ and $q^* =[e^*] = K^*$. We have
\begin{equation}\label{EqnDualCan}
\h^* = (\h^* \cap \k) \oplus (\h^* \cap i\p).
\end{equation}
It follows that $\h^* \cap i\p \subseteq \g^*$ is a Lie triple system and it is
easy to see that $H^*$-orbit through $[e^*]$ coincides with the totally
geodesic exponential image of $\h^* \cap i\p$. Parts~(iii), (iv) and (v) are
obvious from the construction of the dual action. (Note that the slice
representations of~$H_q$ and $H^*_{q^*}$ are given -- on the Lie algebra level
-- by the action of $\h \cap \k$ on the orthogonal complement of $\h \cap \p$
in~$\p$ and on the orthogonal complement of $\h^* \cap i\p$ in~$i\p$,
respectively, and are thus obviously equivalent.)
\end{proof}

\begin{example}\label{ExplHS}\rm
To illustrate the concept, let us describe the dual actions for all connected
reductive algebraic subgroups in the isometry group of the hyperbolic plane.
The hyperbolic plane $M = \HH^2$ and the two-sphere $M^* = \eS^2$ are symmetric
spaces dual to each other. Consider the presentation $\HH^2 = \Sl(2,\R)/\SO(2)$
corresponding to the isometric action of $G = \Sl(2,\R)$ on the upper half
plane as described in Example~\ref{ExplNilp}, where $K = \SO(2)$ is the
stabilizer of the imaginary unit~$i$. Identifying $\eS^2$ with the unit sphere
in $\R^3$, let $G^* = \SO(3)$ and let $K^*$ be the stabilizer of the first
canonical basis vector~$e_1$ of~$\R^3$ under the standard $\SO(3)$-action. We
make the choices  $q = i \in \HH^2$ and $q^* = e_1 \in \eS^2$ for the points
$q,q^*$ as in Theorem~\ref{ThDualAct}. Assume $H \subseteq G$ is a connected
reductive algebraic subgroup. If $H$ is nontrivial, then either $H=G$ or $H$ is
one-dimensional. If $H=G$ then the $H$-action on~$M$ is dual to the
$\SO(3)$-action on $\eS^2$. If $H$ is one-dimensional, we may assume that $\h
\subset \g$ is canonically embedded. This means either $\h = \k$ or $\h \subset
\p$ holds. In the first case we have $H=K$. Then the $H$-action has $i$ as a
fixed point and is dual to the $K^*$-action on $\eS^2$, which has $e_1$ as a
fixed point. In the latter case, since $\HH^2$ is isotropic, we may assume that
$H$ is given by the matrices $b=c=0$ and $ad=1$ with the notation as in
Example~\ref{ExplNilp}; its orbits are the rays in the upper half plane
emanating from~$0$. The totally geodesic orbit $H \cdot i$ is the vertical ray
and $H$ is the group which consists of all transvections along this geodesic. A
dual action on $\eS^2$ is given by choosing $H^*$ as any group of rotations
conjugate to $K^*$ such that the orbit through $q^*$ is a great circle.
Finally, the trivial action on $\eS^2$ is obviously dual to the trivial action
on $\HH^2$.
\end{example}

The example above shows in particular that the action of $\SO(2)$ on the
two-sphere by rotations is dual to two different actions on the hyperbolic
plane. In Section~\ref{Reverse} we will consider another example of an action
on a compact symmetric space which is dual to several different actions.

\begin{remark}\rm
It should be noted that any compact subgroup of a semisimple Lie group is a
reductive algebraic subgroup. Hence the condition that a subgroup $H \subseteq
\Iso(M)$ is reductive algebraic is necessary for the existence of a dual action
of a compact group $H^*$ on $M^*$.
\end{remark}


\section{Polar actions and duality}
\label{PolarDual}


\begin{theorem}\label{ThPolarDual}
Let $M$ be a symmetric space of non-positive curvature without flat factors.
Let $H$ be a connected reductive algebraic subgroup of the isometry group
of~$M$. Let $M^*$ be a compact dual of~$M$ endowed with the dual Riemannian
metric and let $H^*$ be a subgroup of the isometry group of~$M^*$ such that the
$H^*$-action on $M^*$ is dual to the $H$-action on $M$. Then the $H$-action on
$M$ is polar if and only if the $H^*$-action on $M^*$ is polar. In this case, a
section $\Sigma^*$ of the $H^*$-action on $M^*$ is a symmetric space dual to a
section $\Sigma$ of the $H$-action on $M$.  In particular, the $H$-action on
$M$ is hyperpolar if and only if the $H^*$-action on $M^*$ is hyperpolar.
\end{theorem}

We will prove this theorem at the end of this section.  The following is a
useful observation.

\begin{lemma}\label{LmOrthKill}
Let $M$ be a Riemannian manifold and let $\Sigma$ be a connected totally
geodesic submanifold of~$M$. Let $p\in\Sigma$ and let $X$ be a Killing vector
field. Then $X(q) \in N_q\Sigma$ holds for all $q \in \Sigma$ if and only if
$X(p) \in N_p\Sigma$ and $\nabla_v X \in N_p \Sigma$ for all $v \in T_p\Sigma$.
\end{lemma}

\begin{proof}
See \cite[Lemma 5]{dk}.
\end{proof}

\begin{proposition}\label{PropInfCrit}
Let $M$ be a connected Riemannian manifold and let $p \in M$. Let $\s$ be a
linear subspace of $T_pM$ such that the exponential image $\Sigma :=
\exp_p(\s)$ is a totally geodesic submanifold of~$M$. Let $H$ be a connected
closed subgroup of the isotropy group $\Iso(M)_p$. Let $\varrho \colon H \to
\O(T_pM)$ be the orthogonal representation of~$H$ on $T_pM$ where we define
$\varrho(g) \colon T_pM \to T_pM$ to be the differential at~$p$ of the map $x
\mapsto g \cdot x$ for each $g \in H$. Then the following are equivalent:
\begin{enumerate}

\item The submanifold~$\Sigma \subseteq M$ intersects the $H$-orbits
    orthogonally.

\item The linear subspace $\s \subseteq T_pM$ intersects the orbits of
    $\varrho(H)$ orthogonally.

\end{enumerate}
\end{proposition}

\begin{proof}
Let $x$ be an element of the Lie algebra of~$H$. Then for all $q \in M$, the
Killing vector field~$X$ corresponding to~$x$ is given by $X(q) = \left. \frac
d {ds} \right|_{s=0} \left( h_s(q) \right),$ where $h_s$ denotes the isometry
of~$M$ given by the group element $\exp(sx)$, $s \in \R$. Let $\exp_p \colon
T_pM \to M$ denote the Riemannian exponential map of~$M$ at the point~$p$ and
let $v \in \s$. Then we have
\begin{align*}
\nabla_vX &=
 \left. \frac\nabla{\partial t} \frac\partial{\partial s} h_s(\exp_p(tv))
\right\rvert_{s=t=0} = \left. \frac\nabla{\partial s} \frac\partial{\partial t}
h_s(\exp_p(tv)) \right\rvert_{s=t=0} =\\ &= \left. \frac\nabla{\partial s}\left(
\frac\partial{\partial t} \left.h_s(\exp_p(tv))\right\rvert_{t=0} \right)
\right\rvert_{s=0} = \frac\nabla{\partial s} \left. (h_s)_{*p}(v)
\right\rvert_{s=0} = \frac\nabla{\partial s} \left. \varrho(h_s)(v)
\right\rvert_{s=0}.
\end{align*}
From this, it is clear that $\nabla_vX \in N_p\Sigma$ for all $v \in \s$ and
all Killing fields $X$ induced by the $H$-action on $M$ if and only if the
subspace $\s \subseteq T_pM$ intersects the orbits of the $H$-representation on
$T_pM$ orthogonally. Since $X(p)=0$, the statement of the proposition follows
from Lemma~\ref{LmOrthKill}.
\end{proof}

\paragraph{\em Proof of Theorem~\ref{ThPolarDual}}
Assume the $H$-action on $G/K$ is polar. As above, we may assume $q = [e]$ and
$q^* = [e^*]$, where $e \in G$ and $e^* \in G^*$ denote the identity elements.
Identifying as usual the tangent space $T_{[e]}M$ with $\p$, where $\g = \k
\oplus \p$ is a Cartan decomposition, we may identify the tangent space
of~$T_{[e^*]}M^*$ with $i\p$ in the obvious way. Let $\Sigma$ be a section
containing~$[e]$ of the $H$-action on $M$. Since $\Sigma$ is a totally geodesic
submanifold, its tangent space $T_{[e]}\Sigma$ is given by a Lie triple system
$\s \subseteq \p$. It follows from Proposition~\ref{PropInfCrit} that $\s$
intersects the orbits of the linear $H_{[e]}$-action on $T_{[e]}M = \p$
orthogonally. Now consider the $H^*$-action on $M^*$. Obviously, $i\s$
intersects the orbits of $H^*_{[e^*]}$ on $T_{[e^*]}M^* = i\p$ orthogonally and
thus it follows from Proposition~\ref{PropInfCrit} that the totally geodesic
submanifold $\Sigma^*$, which is defined as the exponential image
$\exp_{[e^*]}(i\s)$ of the Lie triple system $i\s \subseteq i\p$, intersects
the orbits of $H^*_{[e^*]}$ in $M^*$ orthogonally. The involution $\sigma^*
\colon \g^* \to \g^*$, defined by $\sigma^*(X) = X$ for $X \in \k$ and
$\sigma^*(Y) = -Y$ for $Y \in i\p$, restricts to an involution of~$\h^*$. Hence
any Killing vector field of~$M^*$ induced by the action of~$H^*$ can be
uniquely written as $X = X' + X''$ such that $X'$ and $X''$ are induced by the
action of~$H$ and $X'([e^*]) = 0$, $\nabla X''([e^*]) = 0$. We have already
seen that $X'(p^*) \perp T_{p^*}\Sigma^*$ for all $p^* \in \Sigma^*$ as $X'$ is
induced by the $H^*_{[e^*]}$-action on~$M^*$. Since $\Sigma$ intersects in
particular the orbit through $[e]$ orthogonally, it follows that $X''([e^*])
\perp \Sigma^*$ with respect to the dual metric. Hence by
Lemma~\ref{LmOrthKill} we get that $X''(p^*) \perp T_{p^*}\Sigma^*$ for all
$p^* \in \Sigma^*$. We have shown that $\Sigma^*$ intersects all $H^*$-orbits
orthogonally. Since $\dim(\Sigma^*)$ equals the cohomogeneity of the
$H^*$-action on $M^*$, it follows by a standard argument that $\Sigma^*$ meets
all $H^*$-orbits. One may proceed in an exactly analogous fashion to show that
the $H$-action on $M$ is polar if the $H^*$-action on $M^*$ is. It is obvious
that the symmetric space $\Sigma^*$ is dual to $\Sigma$. \hfill \qedsymbol


\section{Some applications}
\label{Applications}


We will now state some direct applications of Theorem~\ref{ThPolarDual}.
Henceforth we will always assume that a polar action is nontrivial in the sense
that the orbits are of positive dimension.

\begin{corollary}\label{CorSection}
Let $M$ be an irreducible symmetric space of noncompact type and let $H
\subseteq \Iso(M)$ be a reductive algebraic subgroup acting polarly on~$M$. Let
$\Sigma$ be the section of the $H$-action on $M$. Then $\Sigma$ is isometric to
a product $\R^{n_0} \times \HH^{n_1} \times \ldots \times \HH^{n_k}$.
\end{corollary}

\begin{proof}
Follows from Theorem~\ref{ThPolarDual} and \cite[Theorem 5.4]{polar}.
\end{proof}

Corollary~\ref{CorSection} and \cite[Theorem 5.4]{polar} can be combined by
saying each section of a polar action of a reductive algebraic subgroup of the
isometry group on an irreducible symmetric space is locally isometric to a
Riemannian product of spaces of constant curvature. We can also show that the
Conjecture~\ref{ConjBiliotti} of Biliotti holds also for a large class of
noncompact symmetric spaces if one considers only actions of reductive
algebraic subgroups of the isometry group.

\begin{theorem}\label{ThPolHyp}
Let $M$ be an irreducible symmetric space of type~III such that $\rk(X) \ge 2$.
Let $H \subseteq \Iso(M)$ be a reductive algebraic subgroup acting polarly
on~$M$. Then the sections are flat, i.e.\ the action is hyperpolar.
\end{theorem}

\begin{proof}
Follows from Theorem~\ref{ThPolarDual} and the results of \cite{polar}.
\end{proof}

\begin{theorem}\label{ThExcPolHyp}
Let $M$ be an exceptional symmetric space of type~IV, i.e.\ $X = \LE_6^\C /
\LE_6$, $\LE_7^\C / \LE_7$, $\LE_8^\C / \LE_8$, $\LF_4^\C / \LF_4$, $\LG_2^\C /
 \LG_2$. Let $H \subseteq \Iso(M)$ be a reductive algebraic subgroup acting
polarly on~$M$.  Then the sections are flat, i.e.\ the action is hyperpolar.
\end{theorem}

\begin{proof}
Follows from Theorem~\ref{ThPolarDual} and the results of \cite{polex}.
\end{proof}

\begin{theorem}\label{ThHyper}
Let $M$ be an irreducible Riemannian symmetric space. Assume the reductive
algebraic subgroup $H \subseteq G = \Iso(M)$ acts hyperpolarly and with
cohomogeneity greater than one on~$M$. Then the action of~$H$ on~$M$ is orbit
equivalent to a Hermann action.
\end{theorem}

\begin{proof}
It was shown in~\cite{hyperpolar} that a hyperpolar action on an irreducible
compact symmetric space of cohomogeneity greater than one is orbit equivalent
to a Hermann action. Now assume $M$ is noncompact. Consider a dual action of a
subgroup $H^*$ on a compact dual symmetric space $M^* = G^* / K^*$. Assume that
$\h$ is canonically embedded as in (\ref{EqnCartanDec2}) with respect to a
Cartan decomposition $\g = \k \oplus \p$. We have $\h^* = (\h \cap \k) \oplus
i(\h \cap\p)$. By the result of~\cite{hyperpolar}, it follows that the action
of~$H^*$ on $M^*$ is orbit equivalent to the action of $L^* \subseteq G^*$,
where  $\l^* := \Lie L^* \supseteq \h^*$ and where $\l^*$ is the fixed point
set of some involutive automorphism $\tau \colon \g^* \to \g^*$. The connected
components containing $[e^*]$ of the $H^*$-orbit and of the $L^*$-orbit through
$[e^*]$ agree, thus the projections of~$\h^*$ and $\l^*$ on~$i\p$ agree as
well. Hence we have $\l^* \cap i\p = \h^* \cap i\p$ and $\l^* = (\l^* \cap \k)
\oplus (\l^* \cap i\p)$. It follows that $ \smash \l := \psi^{-1}(\l^*) = (\l^*
\cap \k) \oplus i(\l^* \cap i\p) $ is a subalgebra of~$\g$. Let $\g^* = \l^*
\oplus \m^*$ be the decomposition of $\g^*$ into eigenspaces of~$\tau$. Then we
have the decomposition
\begin{equation}\label{EqnFourDecomp}
\g = (\l^* \cap \k) \oplus i(\l^* \cap i\p) \oplus (\m^* \cap \k) \oplus i(\m^*
\cap i\p).
\end{equation}
Let $\m := (\m^* \cap \k) \oplus i(\m^* \cap i\p)$ and define $\sigma \colon \g
\to \g$ by $\sigma(X) = X$ for $X \in \l$, $\sigma(Y) = -Y$ for $Y \in \m$.
Then $\sigma$ is an involutive automorphism of~$\g$ such that $\l$ is the fixed
point set of~$\sigma$. This follows from the fact that $\sigma$ is just the
restriction $\hat \tau|_\g$ of the automorphism $\hat \tau \colon \g(\C) \to
\g(\C)$ defined by $\hat \tau(X+iY) = \tau(X) +i\tau(Y)$ for $X,Y \in \g^*$, as
can be see from~(\ref{EqnFourDecomp}). Hence the action of the connected
subgroup~$L$ of~$G$ corresponding to~$\l$ is a Hermann action. By construction,
we have $L \supseteq H$ and the $H$-orbits are thus contained in the $L$-orbits
on~$M$. It follows from Proposition~\ref{PropKM}~(i) and
Theorem~\ref{ThDualAct}~(v) that for each $p \in M$ we have $\dim( L \cdot p )=
\dim( H \cdot p )$. We conclude that the $L$-action and the $H$-action on $M$
are orbit equivalent.
\end{proof}


\section{The inverse construction}
\label{Reverse}


Let $M^* = G^* / K^*$ be a symmetric space of compact type and let $G^*$ be the
connected component of the isometry group of~$M^*$. Let $H^*$ be a closed
subgroup of~$G^*$. Let $M=G/K$ be a Riemannian globally symmetric space such
that $M^*$ is a compact dual of~$M$ and such that $G$ is the connected
component of the isometry group of~$M$. Obviously, the action of~$H^*$ is the
dual of an action of a subgroup $H \subseteq G$ on~$M$ if and only if $H^*$ is
conjugate to a subgroup such that (\ref{EqnDualCan}) holds. We have already
seen in Example~\ref{ExplHS} that an action on a compact symmetric space can be
dual to different -- and nonconjugate -- actions on the dual space. In
Example~\ref{ExplHerm} we will look at a specific type of Hermann action from
this point of view. As it will turn out, this action has several dual actions
of various (non-isomorphic) groups on the noncompact dual space, cf.\
\cite{bt1}, where the same phenomenon arises in the context of cohomogeneity
one actions.

\begin{example}\label{ExplHerm}\rm
Let $m,p,q$ be integers such that $1 \le p,q \le m$ and let $n = 2m+1$. We
consider the Hermann action of $H^* = \SO(q) \times \SO(n-q)$ on the
Grassmannian of oriented $p$-dimensional linear subspaces in $\R^n$, which we
denote by ${\rm G}_p(\R^n) = \SO(n) / \SO(p) \times \SO(n-p) = G^* / K^*$. We
will determine all conjugates of $H^*$ for which (\ref{EqnDualCan}) holds. This
is equivalent to determining all types of totally geodesic $H^*$-orbits on $G^*
/ K^*$. The decomposition~(\ref{EqnDualCan}) holds if and only if the
involutions $\sigma$ and $\theta$ commute, where $\sigma,\theta \in \Aut(\g^*)$
are chosen such that $\k = \Lie (K^*) = (\g^*)^\sigma$ and $\h^* =
(\g^*)^\theta$. Define the diagonal matrices
$$
I_{k,n-k} := \left(
             \begin{array}{cc}
               -E_k &  \\
                & E_{n-k} \\
             \end{array}
           \right) \in \O(n),
$$
where $E_k$ denotes the $(k \times k)$-identity matrix. Then $\Ad(I_{k,n-k})
\colon \SO(n) \to \SO(n)$ is an inner automorphism of $\SO(n)$ and we have
$\theta = \Ad(I_{q,n-q})$, $\sigma = \Ad(I_{p,n-p})$. The adjoint
representation $\Ad \colon \SO(n) \to \Aut(\so(n))$ is faithful since $n$ is
odd. Thus for $A,B \in \SO(n)$ we have $\Ad(AB) = \Ad(BA)$ if and only if $A$
and $B$ commute. Now let $A_g = g \cdot I_{q,n-q} \cdot g^{-1}$ for $g \in G^*$
and let $B = I_{p,n-p}$. Then the connected component of the fixed point set of
$\theta_g := \Ad(A_g)$ is $gH^*g^{-1}$.

Assume now that $g \in  G^*$ is such that $\sigma \circ \theta_g = \theta_g
\circ \sigma$. We will determine the type of the $H^*$-orbit through $[g^{-1}]
= g^{-1}K^*$, i.e.\ we compute the isotropy subgroup $$H^*_{[g{^-1}]} = \{ h
\in H^* \mid h g^{-1} K^* = g^{-1} K^* \},$$ which is conjugate to $g H^* g^{-1}
\cap K^*$. If the matrices $A_g$ and $B$ commute, then there is a decomposition
$\R^n = V_{00} \oplus V_{01} \oplus V_{10} \oplus V_{11}$ such that $A_g
\rvert_{V_{\varepsilon\delta}} = (-1)^\varepsilon \cdot
\id_{V_{\varepsilon\delta}}$ and $B\rvert_{V_{\varepsilon\delta}} = (-1)^\delta
\cdot \id_{V_{\varepsilon\delta}}$. Let $r := \dim(V_{00})$. Then we have $0
\le r \le \min(p,q)$ and $r$ attains all values in this range for suitable $g
\in G^*$. We obtain $gHg^{-1} \cap K^* =$
\begin{equation*}\label{EqnTGOrbits}
=\left\{
  \begin{array}{ll}
    \SO(r) \times \SO(p-r) \times \SO(q-r) \times \SO(n-p-q+r), & \hbox{if $1 \le r < \min(p,q)$;} \\
    \SO(p) \times \SO(q) \times \SO(n-p-q), & \hbox{if $r=0$;} \\
    \SO(p) \times \SO(q-p) \times \SO(n-q), & \hbox{if $r=p<q$;} \\
    \SO(q) \times \SO(p-q) \times \SO(n-p), & \hbox{if $r=q<p$;} \\
    \SO(p) \times \SO(n-p), & \hbox{if $r=p=q$.}
  \end{array}
\right.
\end{equation*}
Note that the value of~$r$ determines the orbit type of the $H^*$-orbit
through~$[g^{-1}]$. Finally, we can determine the conjugacy classes of
connected closed subgroups~$H$ of $G = \SO_0(p,n-p)$ with the property the
$H^*$-action on~$M^*$ is dual to the $H$-action on~$M = \SO_0(p,n-p) / \SO(p)
\times \SO(n-p)$. In case $p<q$ they are given by
\begin{align*}
    &\SO_0(p,n-p-q) \times \SO(q); \\
    &\SO_0(r,q-r) \times \SO_0(p-r,n-p-q+r),\ 1\le r < p; \\
    &\SO_0(p,q-p) \times \SO(n-q).
\end{align*}
If $q < p$ we obtain
\begin{align*}
    &\SO_0(p,n-p-q) \times \SO(q); \\
    &\SO_0(r,q-r) \times \SO_0(p-r,n-p-q+r),\ 1\le r < q; \\
    &\SO_0(p-q,n-p) \times \SO(q).
\end{align*}
Finally, in case $p=q$ they are
\begin{align*}
    &\SO_0(p,n-2p) \times \SO(p); \\
    &\SO_0(r,p-r) \times \SO_0(p-r,n-2p+r),\ 1\le r < p; \\
    &\SO(p) \times \SO(n-p).
\end{align*}
This example nicely illustrates how one action on a compact symmetric space can
be the dual of several nonconjugate actions on the noncompact dual symmetric
space. In this case, the data determining the various actions on the noncompact
space is encoded into just one action on the compact dual.  This imbalance is
made up for by the fact that the various actions on the noncompact space are of
a simpler structure in that the whole space is equivariantly diffeomorphic to
the normal bundle of a totally geodesic orbit, which is not true for the dual
action on the compact space.
\end{example}


\section{Polar actions on real hyperbolic space}
\label{PAHS}


Using duality and the classification of polar actions on compact rank one
symmetric spaces by Podest\`{a} and Thorbergsson~\cite{pth1}, we will obtain a
classification of polar actions of reductive algebraic subgroups of the
isometry group on noncompact rank one symmetric spaces.  We start with real
hyperbolic space. Note that polar actions on real hyperbolic space have been
classified by Bingle Wu \cite{wu} without assuming that the action is induced
by a reductive algebraic subgroup of the isometry group. However, the duality
method we are using here is not restricted to spaces of constant curvature and
we will obtain classification results also for the other noncompact rank one
symmetric spaces in Sections~\ref{PACH}--\ref{PACayHP}.

\begin{theorem}\label{ThHyperbolic}
Let $H \subseteq G := \SO_0(1,n)$ be a connected reductive algebraic subgroup.
Then the $H$-action on hyperbolic space $\HH^n = \SO_0(1,n) / \SO(n)$ is polar
if and only if one the following is true.
\begin{enumerate}

\item The subgroup $H$ is conjugate to $\SO_0(1,k) \times L \subseteq
    \SO(1,n)$, $k= 1, \ldots n$, where $L \subseteq \SO(n-k)$ is a subgroup
    acting polarly on $\R^{n-k}$.

\item The subgroup $H$ is conjugate to a subgroup $L \subseteq \SO(n)$
    acting polarly on $\R^{n}$.

\end{enumerate}
In case~(i) the $H$-action has a totally geodesic orbit isometric to $\HH^k$,
in case~(ii) it has a fixed point.
\end{theorem}

\begin{proof}
Assume $H$ acts polarly on $\HH^n$. It follows from Theorem~\ref{ThPolarDual}
that there is a dual polar action of a compact connected group $H^* \subseteq
\SO(n+1)$ on the sphere $\eS^n = \SO(n+1) / \SO(n)$ and we may assume that the
orbit $H^* \cdot [e^*]$ is a totally geodesic submanifold of~$\eS^n$ such that
(\ref{EqnDualCan}) holds. We may identify the sphere $\eS^n$ with a sphere
around the origin in the Euclidean space $\R^{n+1}$ and assume the action
of~$H^*$ on $\eS^n$ is given by restriction of the standard representation
of~$\SO(n+1)$. The totally geodesic orbit $H^* \cdot [e^*]$ is then given by
the intersection of~$\eS^n$ with some linear subspace $V \subseteq \R^{n+1}$.
This space $V$ is an invariant subspace for the $H^*$-action on~$\R^{n+1}$ and
the orbit $H^* \cdot [e^*]$ is a great sphere $\eS^k \subseteq \eS^n$, where $0
\le k \le n$. We may assume that $V$ is spanned by the first $k+1$ canonical
basis vectors of $\R^{n+1}$. Let $V^\perp$ be the orthogonal complement of~$V$
in $\R^{n+1}$. It follows from (\ref{EqnDualCan}) that $H^*$ is of the form
$\SO(k+1) \times L$, where the first factor is standardly embedded and where
the factor $L$ is contained in the centralizer of the first factor. Hence
$\SO(k+1)$ acts by the standard representation on~$V$ and trivially
on~$V^\perp$, while the second factor acts trivially on~$V$. Since polar
representations act polarly on their invariant submodules~\cite{dadok}, it
follows that $L \subseteq \SO(n-k)$ is a compact connected group whose action
on~$\R^{n-k}$ is polar. It follows that $H$ is as in item~(i) if $k \ge 1$ and
as in item~(ii) if $k=0$ and the corresponding orbit $H \cdot [e]$ is isometric
to $\HH^k$ in case $k=1,\ldots,n$, or a point in case $k=0$.

Conversely, it is easy to see that the actions as described in (i) and (ii)
have polar dual actions and are hence polar by Theorem~\ref{ThPolarDual}.
\end{proof}

Let us compare the above theorem with the result of Bingle
Wu~\cite[Theorem~3.3]{wu}, which is very similar and which was proven without
assuming that the subgroup of $\SO_0(1,n)$ given by the action is reductive
algebraic. Instead it was assumed in \cite{wu} that the principal orbits of the
polar action are {\em full} isoparametric submanifolds of~$\HH^n$, i.e.\ they
are not contained in a totally umbilic submanifold; however it is shown in
\cite[Corollary~2.6]{wu} that such an action always has a totally geodesic
orbit and that it is orbit equivalent to an action of some subgroup of
$\SO_0(1,n)$ conjugate to $\SO_0(1,k) \times L$, where $L$ is a compact Lie
group. In particular, the action is orbit equivalent to the action of a
reductive algebraic subgroup of the isometry group. It follows from
\cite[Theorem~3.1]{wu} that polar actions on $\HH^n$ whose principal orbits are
not full are given by polar actions on a totally umbilic submanifold~$U$
of~$\HH^n$. Such a totally umbilic submanifold $U \subset \HH^n$ is either a
totally geodesic $\HH^k \subset \HH^n$, a round sphere, or a submanifold which
is flat in its induced metric. In the last case, it follows from
\cite[Theorem~3.1]{wu} that the action is as described in
Example~\ref{ExplEHP}. The case where the principal orbits of an action are
contained in a round sphere corresponds to the case of actions with a fixed
point.

\begin{corollary}\label{CorFullAlgRed}
The principal orbits of a polar action on $\HH^n$ are full isoparametric
submanifolds of~$\HH^n$ if and only if the action is orbit equivalent to an
action of a reductive algebraic subgroup of $\Iso(\HH^n)$ such that a dual
action on $\eS^n$ is polar with full isoparametric submanifolds of~$\eS^n$  as
principal orbits.
\end{corollary}

\begin{proof}
It suffices to observe that the orbits of the orbits of the $H$-action on
$\HH^n$ are full if and only if the action is as described in part (i) of
Theorem~\ref{ThHyperbolic} and such the representation of~$L$ on $\R^{n-k}$
does not have any nonzero fixed vectors.
\end{proof}


\section{Polar actions on complex hyperbolic space}
\label{PACH}


To study polar actions on complex hyperbolic space, we will proceed in a
similar fashion as in the proof of Theorem~\ref{ThHyperbolic}. Polar actions on
complex projective space have been classified by Podest\`{a} and Thorbergsson
\cite{pth1}. Their result says that polar actions on $\C\P^n$ are orbit
equivalent to actions given by the following construction. Let $(G,K) =
(\Pi_{\mu = 1}^\nu G_\mu, \Pi_{\mu = 1}^\nu K_\mu)$ be a Hermitian symmetric
pair such that $G_\mu/K_\mu$ are irreducible compact Hermitian symmetric
spaces. Let $\g_\mu = \k_\mu \oplus \p_\mu$ be the corresponding
decompositions. On each $\p_\mu$ there exists a complex structure $J_\mu$,
which is unique up to sign, and we may identify $\p = \p_1 \oplus \ldots \oplus
\p_\nu$ with $\C^d$, where $d$ is the complex dimension of $G/K$. Then the
action of the group $K$ on $\C^d$ thus defined descends to a polar action on
$\C\P^{d-1}$ and conversely \cite[Theorem 3.1]{pth1}, every polar action on
$\C\P^{d-1}$ is orbit equivalent to such an action. We will say that a
representation of a compact Lie group $K$ on $\C^d$ is {\em induced by a
Hermitian symmetric space} if the $K$-action on $\C^d$ is given by the
construction just described. The irreducible Hermitian symmetric spaces of
compact type are
\begin{align}\label{EqnHermSymm}
\begin{array}{c}
\SU(p+q) / \SUxU pq,\quad
\SO(k+2)/\SO(2)\times\SO(k), \quad
\Sp(k)/\U(k),\\
\SO(2k)/\U(k), \quad
\LE_6 / \U(1)\cdot\Spin(10), \quad
\LE_7 / \U(1) \cdot \LE_6,
\end{array}
\end{align}
see \cite[Ch. X, \S6.3]{helgason}. Recall that there is some overlap between
the different types in (\ref{EqnHermSymm}), cf.\ \cite[Ch.\ X,
\S~6.4]{helgason}.

\begin{theorem}\label{ThCHn}
Let $H \subset G := \SU(1,n)$ be a reductive algebraic subgroup. Then the
$H$-action on complex hyperbolic space $\C\HH^n = \SU(1,n) / \SUxU1n $ is polar
if and only if it is orbit equivalent to one of the following actions.
\begin{enumerate}

\item The action of $\eS(\U(1,k) \times L) \subseteq \SU(1,n)$, $k= 1,
    \ldots n$, where $L \subseteq \U(n-k)$ is a subgroup whose action on
    $\C^{n-k}$ is induced by a Hermitian symmetric space.

\item The action of $\eS((\U(1) {\cdot} \SO_0(1,k)) \times L) \subseteq
    \SU(1,n)$, $k= 1, \ldots n$, where $L \subseteq \U(n-k)$ is a subgroup
    whose action on $\C^{n-k}$ is induced by a Hermitian symmetric space.

\item The action of a subgroup $L \subseteq \SUxU1n \cong \U(n)$ whose
    action on $\C^{n}$ is induced by a Hermitian symmetric space.

\end{enumerate}
In case (i) the $H$-action on $\C\HH^n$ has a totally geodesic orbit isometric
to $\C\HH^k$, in case (ii) it has a totally geodesic orbit isometric to
$\HH^k$, in case (iii) it has a fixed point.
\end{theorem}

\begin{proof}
Let $G^* = \SU(n+1)$ and let $K^* = \SUxU1n$. Proceeding as in the proof of
Theorem~\ref{ThHyperbolic}, we may assume $H^* \cdot [e^*]$ is a totally
geodesic orbit and $H^*$ acts polarly on $\C\P^n = G^*/K^*$ by
Theorem~\ref{ThPolarDual}. Using the natural projection map $\eS^{2n+1} \to
\C\P^n$, $(z_1, \ldots, z_{n+1}) \mapsto [z_1 \colon \ldots \colon z_{n+1}]$,
we may identify the points in $\C\P^n$ with the fibers of the Hopf fibration on
the unit sphere $\eS^{2n+1} \subset \C^{n+1}$ around the origin, i.e.\ with
orbits of unit vectors in $\C^{n+1}$ under multiplication with complex scalars
of unit norm. Replacing $H^*$ with a group whose action on $\C\P^n$ is orbit
equivalent to the $H^*$-action, if necessary, we may assume the subgroup $H^*
\subseteq \SU(n+1)$ is such that the action of $\U(1) \cdot H^*$ on $\C^{n+1}$
is induced by a Hermitian symmetric space~\cite{pth1}. As it was shown
in~\cite{wolf}, the totally geodesic submanifolds of positive dimension in
$\C\P^n$ are isometric to either $\C\P^k$ where $k = 1,\ldots,n$ or $\R\P^k$
where $k = 1,\ldots,n$ and any such totally geodesic submanifold is conjugate
by an isometry to the standard embedding of $\SU(k+1)/\SUxU1k$ or
$\SO(k+1)/\eS(\O(1)\times\O(k))$ into $\SU(n+1)/\SUxU1n$.

Let us first consider the case where $H^* \cdot [e^*]$ is isometric to
$\C\P^k$. It follows that the action of $H^*$ on $\C^{n+1}$ leaves a complex
$(k+1)$-dimensional linear subspace invariant, i.e.\ there is an
$H^*$-invariant decomposition $\C^{n+1} = \C^{k+1} \oplus \C^{n-k}$, where
$H^*$ acts irreducibly on the first summand $\C^{k+1}$. Thus the action of
$H^*$ on $\C^{k+1}$ is induced by an irreducible Hermitian symmetric space of
complex dimension~$k+1$ such that the action induced on $\C\P^k$ is transitive.
Since any compact subgroup of $\SU(k+1)$ acting transitively on $\C\P^k$ also
acts transitively on the unit sphere in $\C^{k+1}$ by~\cite{oniscik}, the
action of $H^*$ on $\C^{k+1}$ is induced by a Hermitian symmetric space of rank
one, hence by $\C\P^{k+1}$. Furthermore, the action of $H^*$ on $\C^{n-k}$ is
induced by some complex $(n-k)$-dimensional Hermitian symmetric space $Q/L$.
This shows that the $H$-action on $\C\HH^n$ is as described in item~(i).

Let us now assume $H^* \cdot [e^*]$ is isometric to $\R\P^k$. Since the
embedding $\R\P^k \subset \C\P^n$ is given by the standard embedding of
$\SO(k+1)/\eS(\O(1)\times\O(k))$ into $\SU(n+1)/\SUxU1n$ and the span of an
orbit of a representation is an invariant subspace, we may assume that we have
an $H^*$-invariant decomposition $\C^{n+1} = \C^{k+1} \oplus \C^{n-k}$. The
action of $H^*$ on $\C^{n-k}$ is induced by a complex $(n-k)$-dimensional
Hermitian symmetric space~$Q/L$.  The action of $H^*$ on the first summand
$\C^{k+1}$ is obviously irreducible and $H^* / H^* \cap K^*$ is a -- possibly
non-effective -- homogeneous presentation of $\SO(k+1)/\eS(\O(1) \times
\O(k))$. Thus the action of $H^*$ on $\C^{k+1}$ is induced by an irreducible
Hermitian symmetric space of real dimension~$2(k+1)$ whose isotropy group
contains a normal factor locally isomorphic to $\SO(k+1)$. From
(\ref{EqnHermSymm}) we deduce that the action of $H^*$ on $\C^{k+1}$ is induced
by the complex quadric $\SO(k+3) / \SO(2) \times \SO(k+1)$. Thus the $H$-action
on $\C\HH^n$ is as described in item~(ii).

It was shown in \cite{dk} that polar actions with a fixed point on $\C\HH^n$
are exactly the actions as described in item~(iii).

Now let $H \subseteq \SU(1,n)$ be a closed connected subgroup as described in
parts (i) or (ii) of the theorem. Then obviously the $H$-action on $\C\HH^n$
has a totally geodesic orbit which can be identified with $\eS(\U(1,k) \times
L) / \eS(\U(1) \times \U(k) \times L)$ or $\eS(\U(1) \cdot \SO_0(1,k) \times L)
/ \eS(\U(1) \cdot \SO(k) \times L)$ where in both cases $L$ is a compact Lie
group and we see that the group $H^*$ is of the form $\eS(\U(k+1) \times L)$ or
$\eS(\U(1) \cdot \SO(k+1) \times L)$. In view of Theorem~\ref{ThPolarDual} and
\cite[Proposition 2A.1]{pth1}, it suffices to show that the action of $\U(1)
\cdot H^*$ on $\C^{n+1}$ is induced by a Hermitian symmetric space. The action
of $\U(1) \cdot L$ on $\C^{n-k}$ is induced by a Hermitian symmetric space
$Q/L$ by the hypothesis and we see that the action of $\U(1) \cdot H^*$ on
$\C^{n+1}$ is induced by the Hermitian symmetric space $(\SU(k+2) /
\SUxU1{k+1}) \times (Q/L)$ or $(\SO(k+3) / \SO(2) \times \SO(k+1)) \times
(Q/L)$.
\end{proof}


\section{Polar actions on quaternionic hyperbolic space}
\label{PAQH}


Let us first briefly review the results of \cite[Theorem 4.1]{pth1}. Let
$$(G,K)= (\Pi_{\mu = 1}^\nu G_\mu, \Pi_{\mu = 1}^\nu K_\mu)$$ be a symmetric pair such
that $G_\mu/K_\mu$ are compact quaternion-K\"{a}hler symmetric spaces. Let $\g_\mu
= \k_\mu \oplus \p_\mu$ be the corresponding decompositions. Then we have
$K_\mu = \Sp(1) \cdot H_\mu$, where both factors are normal subgroups. Using
the quaternionic structure induced by $\ad(\sp(1))$ on $\p_\mu$, we may
identify $\p = \p_1 \oplus \ldots \oplus \p_\nu$ with $\H^d$, where $d =
\frac14\dim (G/K)$ and where $H = H_1 \times \ldots \times H_\nu$ acts linearly
on $\H^d = \R^{4d}$ in such a fashion that the $H$-action commutes with the
$\Sp(1)$-action defined by right multiplication with the unit quaternions in
$\Sp(1)$. This action of~$H$ on $\H^d$ descends to an action on $\H\P^{d-1}$
and we will say that a representation of a compact Lie group $K$ on $\H^d$ is
{\em induced by a product of $\nu$~quaternion K\"{a}hler symmetric spaces} if the
$H$-action on $\H^d$ is given by the above construction; we say that it is {\em
induced by a quaternion-K\"{a}hler symmetric space} if $\nu=1$. Under the
additional assumption that at most one of the factors $G_1/K_1, \ldots,
G_\nu/K_\nu$ is of rank greater than one, the action of~$H$ on $\H\P^{d-1}$
just defined is polar and conversely, every polar action on $\H\P^{d-1}$ is
orbit equivalent to an action of some group $K \subseteq \Sp(d)$ whose action
on $\C^d$ is induced by a product of $\nu$~quaternion K\"{a}hler symmetric spaces
where at most one of the factors is of rank greater than one. The compact
quaternion-K\"{a}hler symmetric spaces are the following:
\begin{align*}\label{EqnQKSymm}
\begin{array}{c}
\Sp(n+1) / \Sp(1) \cdot \Sp(n),\quad
\SU(n+2) / \SUxU2n, \\
\SO(n+4) / \SO(4) \times \SO(n),\quad
\LG_2 / \SO(4), \quad
\LF_4 / \Sp(1) \cdot \Sp(3), \\
\LE_6 / \Sp(1) \cdot \SU(6), \quad
\LE_7 / \Sp(1) \cdot \Spin (12), \quad
\LE_8 / \Sp(1) \cdot \LE_7,
\end{array}
\end{align*}
see~\cite[Ch.\ 14 E]{besse}.

\begin{theorem}\label{ThQHn}
Let $H \subset G := \Sp(1,n)$ be a reductive algebraic subgroup. Then the
$H$-action on complex hyperbolic space $\H\HH^n = \Sp(1,n) / \Sp(1) \times
\Sp(n)$ is polar if and only if it is orbit equivalent to one of the following
actions.
\begin{enumerate}

\item The action of $\Sp(1,k) \times \Sp(n_1) \times \ldots \times
    \Sp(n_\nu) \times L \subseteq \Sp(1,n)$, where $L$ is a subgroup
    of~$\Sp(m)$ whose action on $\H^{m}$ is induced by a quaternion K\"{a}hler
    symmetric space, where $1 \le k \le n$ and $k + n_1 + \ldots + n_\nu +
    m = n$.

\item The action of $\U(1,k) \times \Sp(n_1) \times \ldots \times
    \Sp(n_\nu) \subseteq \Sp(1,n)$, where $1 \le k \le n$ and $k + n_1 +
    \ldots + n_\nu = n$.

\item The action of $(\Sp(1) \cdot \SO_0(1,k)) \times \Sp(n_1) \times
    \ldots \times \Sp(n_\nu) \subseteq \Sp(1,n)$, where $1 \le k \le n$ and
    $k + n_1 + \ldots + n_\nu = n$.

\item The action of $\Sp(1) \times L \subseteq \Sp(n)$ where $L$ is a
    subgroup whose action on $\H^{n}$ is induced by a product of quaternion
    K\"{a}hler symmetric spaces where at most one of the factors is of rank
    greater than one.
\end{enumerate}
The $H$-action on $\H\HH^n$ has a totally geodesic orbit isometric to $\H\HH^k$
in case (i), $\C\HH^k$ in case (ii), $\HH^k$ in case (iii), a point in case
(iv).
\end{theorem}

\begin{proof}
The proof is mostly analogous to the proof of Theorem~\ref{ThCHn}. As the case
of polar actions on $\H\HH^n$ with a fixed point was settled in \cite{dk},
where it was shown they are all orbit equivalent to the actions as described in
item~(iv), we may restrict ourselves to actions without fixed points.

According to \cite{wolf}, the totally geodesic submanifolds of positive
dimension in $\H\P^n$ are isometric to $\eS^k$, $k = 1,\ldots, 4$, or $\R\P^k$,
$\C\P^k$, $\H\P^k$, where $k = 2,\ldots,n$ and any two homeomorphic totally
geodesic submanifolds are conjugate by an isometry. Hence the totally geodesic
subspaces $\R\P^k$, $\C\P^k$, $\H\P^k$ are given by the standard embeddings
$\SO(k) \subset \SU(k) \subset \Sp(k) \subseteq \Sp(n)$ for $k = 2,\ldots,n$
and also the totally geodesic spheres $\eS^1 = \R\P^1  \subset \eS^2 = \C\P^1
\subset \eS^3 \subset \eS^4 = \H\P^1 \subseteq \H\P^n$ are given by the
standard embeddings.

First assume the totally geodesic orbit $H^* \cdot [e^*]$ is isometric to
$\H\P^k$, where $k \in \{1,\ldots,n\}$. By an analogous argument as in the
proof of Theorem~\ref{ThCHn}, we see that $H = \Sp(1,k) \times L$ and $H \cap K
= \Sp(1) \times \Sp(n) \times L$, where $L \subseteq \Sp(n-k)$. Since $H^*$ is
of the form $\Sp(k+1) \times L$, it follows from \cite[Theorem 4.1]{pth1} that
the action of~$L$ on $\H^{n-k}$ is induced by a product of quaternion K\"{a}hler
symmetric spaces where at most one factor is of rank greater than one. Hence we
have an action as described in item~(i) of the theorem.

Now consider the case where the totally geodesic orbit $H^* \cdot [e^*]$ is
isometric to $\C\P^k$ or $\R\P^k$, where $1 \le k \le n$. An argument analogous
as in the proof of Theorem~\ref{ThCHn} shows that the $H$-action on~$M$ is as
described in items~(ii) or (iii).

It remains the case where the totally geodesic orbit $H^* \cdot [e^*]$ is
isometric to a three-sphere, which is a great sphere in a standardly embedded
totally geodesic $\eS^4 = \H\P^1 \subseteq \H\P^n$. It follows that the action
of $H^*$ on $\H^{n+1}$ leaves a quaternionic subspace isomorphic to $\H^2$
invariant on which $H^*$ acts by the standard $\Sp(2)$-representation. But this
action does not have a three-dimensional orbit and thus we have arrived at a
contradiction.

Conversely, it is easy to see by an analogous argument as in the proof of
Theorem~\ref{ThCHn} that the actions as described in parts~(i) to (iv) have
polar dual actions on $\H\P^n$ and are thus polar by Theorem~\ref{ThPolarDual}.
Indeed, the polar dual action on $\H\P^n$ is induced by $(\Sp(k+2) / \Sp(1)
\times \Sp(k+1)) \times (Q/L)$ in case~(i), it is induced by $(\SU(k+3) /
\SUxU2{k+1}) \times (Q/L)$ in case~(ii), it is induced by $(\SO(k+5) / \SO(4)
\times \SO(k+1)) \times (Q/L)$ in case~(iii) and induced by $(\Sp(2) / \Sp(1)
\times \Sp(1)) \times (Q/L)$ in case~(iv), where $Q/L$ is in each case a
product of quaternion K\"{a}hler symmetric spaces.
\end{proof}


\section{Polar actions on the Cayley hyperbolic plane}
\label{PACayHP}


In this section we classify polar actions on the Cayley hyperbolic plane
$\Ca\HH^2 = \LF_{4(-20)} / \Spin(9)$ by reductive algebraic subgroups of the
isometry group. Polar actions on the Cayley plane $\Ca\P^2 = \LF_{4} /
\Spin(9)$ -- which is the compact dual of~$\Ca\HH^2$ -- were classified by
Podest\`{a} and Thorbergsson, see \cite[Theorem~5.1]{pth1}. Their result is the
following. A connected subgroup $H$ of $\LF_4$ acts polarly and with a fixed
point on $\Ca\P^2$ if and only if it is conjugate to one of $\Spin(9)$,
$\Spin(8)$, $\SO(2) \cdot \Spin(7)$, or $\Spin(3) \cdot \Spin(6)$; it acts
polarly and without fixed point if and only if it is conjugate to one of
$\Sp(3) \cdot \Sp(1)$, $\Sp(3) \cdot \U(1)$, $\Sp(3)$, or $\SU(3) \cdot
\SU(3)$, where the first three groups act with cohomogeneity one. In fact, the
actions of the first three groups are orbit equivalent. The action of the last
group $\SU(3) \cdot \SU(3)$ is of cohomogeneity two.

Let us also review the results of \cite{wolf} concerning totally geodesic
submanifolds of $\Ca\P^2$. All totally geodesic submanifolds of positive
dimension in  $\Ca\P^2$ are homothetic to one of $\eS^1$, $\eS^2, \ldots,
\eS^8$, $\R\P^2$, $\C\P^2$, $\H\P^2$, or $\Ca\P^2$. Moreover, any two
homeomorphic totally geodesic subspaces are conjugate by an isometry.

Our proof of the theorem below does not proceed analogously as for
Theorems~\ref{ThHyperbolic}, \ref{ThCHn} and \ref{ThQHn}, instead we will
consider the polar actions on $M^*$ and classify all actions dual to them.

\begin{theorem}\label{ThPolarHypCay}
Let $H \subset \LF_{4(-20)}$ be a connected reductive algebraic subgroup. Then
the $H$-action on the Cayley hyperbolic plane $\Ca\HH^2 = \LF_{4(-20)} /
\Spin(9)$ is polar if and only if $H$ is conjugate to one of the subgroups $H$
as given in Table~\ref{TbPolarHypCay}.
\begin{table}[h]\rm
\hspace{7em}
\begin{tabular}{|c|c|c|c|}
\hline
\stru $H$ & ${\hbox{cohomo-}}\atop{\hbox{geneity}}$ & ${\hbox{totally}}\atop{\hbox{geodesic orbit}}$   \\
\hline\hline
\str $\Spin(9)$ & $1$ & \{pt.\} \\
\hline
\str $\Spin(1,8)$ & $1$ & $\HH^8$ \\
\hline
\str $\Spin(8)$ & $2$ & \{pt.\} \\
\hline
\str $\Spin(1,7)$ & $2$ & $\HH^7$ \\
\hline
\str $\SO(2) \cdot \Spin(7)$ & $2$ & \{pt.\} \\
\hline
\str $\SO_0(1,1) \cdot \Spin(7)$ & $2$ & $\R$ \\
\hline
\str $\SO(2) \cdot \Spin(1,6)$ & $2$ & $\HH^6$ \\
\hline
\str $\Spin(3) \cdot \Spin(6)$ & $2$ & \{pt.\} \\
\hline
\str $\Spin(1,2) \cdot \Spin(6)$ & $2$ & $\HH^2$ \\
\hline
\str $\Spin(3) \cdot \Spin(1,5)$ & $2$ & $\HH^5$ \\
\hline
\triplestru \begin{tabular}{l}
   $\Sp(1,2) \cdot \Sp(1)$ \\
   $\Sp(1,2) \cdot \U(1)$ \\
   $\Sp(1,2)$ \\
\end{tabular} & $1$ & $\H\HH^2$ \\
\hline
$\stru \SU(1,2) \cdot \SU(3)$ & $2$ & $\C\HH^2$ \\
\hline
\end{tabular}
\bl\caption{\rm Polar actions on the Cayley hyperbolic
plane}\label{TbPolarHypCay}
\end{table}\rm
\end{theorem}
For each action in Table~\ref{TbPolarHypCay} the cohomogeneity and the
(uniquely defined) type of totally geodesic orbit is given. Actions which are
orbit equivalent to each other are listed in consecutive rows of the table
without separating horizontal lines.

\begin{proof}
The case of a polar action on $\Ca\HH^2$ with a fixed point was already settled
in \cite{dk}, the result is that the subgroups of $\Spin(9)$ acting polarly
with a fixed point on $M = \Ca\HH^2$ are exactly the same as those acting
polarly with a fixed point on $M^* = \Ca\P^2$. Hence we may assume for the
remaining part of the proof that the action of~$H$ on $M$ has a totally
geodesic orbit of positive dimension. Let $G^*$ be the compact Lie group of
type $\LF_4$ and let $K^* = \Spin(9)$. Let $\g^* = \k^* \oplus \p^*$ be the
usual decomposition. We will determine all closed connected subgroups $H^*
\subset G^*$ acting polarly on $M^* = G^* / K^*$ and such that $\h^* = (\h^*
\cap \k^*) \oplus (\h^* \cap \p^*)$, proceeding similarly as in
Example~\ref{ExplHerm}. As we do not need to consider fixed points, we may
ignore the cases where $\h^* \subseteq \k^*$. As pointed out above, the
conjugacy classes of connected closed subgroups $H^*
 \subset G^*$ acting polarly on $M^*$ have been determined in~\cite{pth1}.

Let us start with the isotropy action, i.e.\ the action of $\Spin(9)$ on $M^* =
\LF_4 / \Spin(9)$. As this is a Hermann action, the desired information can be
read off from \cite[Table~1]{z2z2}. We see that, apart from the fixed point of
this action, the only other type of totally geodesic orbit which occurs is
$\eS^8 = \Spin(9) / \Spin(8)$; since we are looking at an action of
cohomogeneity one, there are only two singular orbits. It follows that there is
exactly one subgroup of $\LF_4$ conjugate to $\Spin(9)$ whose orbit through
$[e^*]$ is totally geodesic and this action is dual to the action of
$\Spin(1,8)$ on $M^*$.

Let us now consider the proper subgroups of $\Spin(9)$ which act polarly on
$\Ca\P^2$. Consider the action of $H^* = \Spin(8)$ on~$\Ca\P^2$. If $\h^* =
(\h^* \cap \k^*) \oplus (\h^* \cap \p^*)$ and $\h^* \not\subset \k^*$ then it
follows that $\h^* \cap \k^* = \spin(7)$ by the classification of symmetric
spaces~\cite{helgason}, since $H^* / H^* \cap K^*$ is a rank one symmetric
space in this case. From the argument on the isotropy action of $\Spin(9)$
above, we see that a suitable conjugate of $\Spin(8) \subset \LF_4$ actually
has a totally geodesic orbit of type $\eS^7$. The other cases are similar.

We will now consider the Hermann action of $\Sp(3) \cdot \Sp(1)$ on $M^*$. It
follows from \cite[Table~1]{z2z2} that it has only one totally geodesic orbit
which is homothetic to $\H\P^2 = \Sp(3) / \Sp(1) \times \Sp(2)$. This action is
obviously dual to the action of $\Sp(1,2) \cdot \Sp(1)$ on $\Ca\HH^2$; the
$\Sp(1)$-factor acts trivially on the totally geodesic orbit and we see that
more generally the action of $\Sp(1,2) \cdot L$ on $\Ca\HH^2$ is dual to the
action of $\Sp(3) \cdot L$, where $L \subseteq \Sp(1)$ is a closed connected
subgroup.

Finally, it remains to consider the action of $H^* = \SU(3) \cdot \SU(3)$ on
$M^*$. Note that the two isomorphic $\SU(3)$-factors are not conjugate by any
automorphism of $\LF_4$. In fact, the two simple factors correspond to two
subsystems both of type~$\LA_2$ inside the root system of type $\LF_4$, which
are orthogonal to each other, one consisting of long roots, the other
consisting of short roots, see e.g.\ \cite[Ch.~ \S3.11]{oniscikBook}. Assume we
have $\h^* = (\h^* \cap \k^*) \oplus (\h^* \cap \p^*)$. Then $(H^*, H^* \cap K^*)$
is a symmetric pair such that $H^* / H^* \cap K^*$ is a rank one symmetric
space, the only possibility being $\h^* \cap \k^* \cong \suxu12 \oplus \su(3)$.
In fact, it has been shown in \cite[Lemma~2B.3]{pth1} that the action under
consideration has a totally geodesic orbit of type $\C\P^2$. This shows that
the action of $\SU(1,2) \cdot \SU(3)$ on $M$ is dual to the $H^*$-action on
$M^*$. Furthermore, there are no other totally geodesic orbits. To see this, it
suffices to note that only one of the $\SU(3)$-factors is conjugate to a
subgroup of $\Spin(9) \subset \LF_4$, namely the one whose roots are short.
\end{proof}

\begin{corollary}
A polar action of a compact Lie group on $\Ca\P^2$ has a totally geodesic orbit.
\end{corollary}

\begin{proof}
Follows from the proof of Theorem~\ref{ThPolarHypCay}.
\end{proof}


\section{Conclusion}
\label{Conclusion}


The method of dual actions turns out to be a useful tool for the study of
isometric Lie group actions on symmetric spaces of the noncompact type. Under
the hypothesis that the action is induced by a reductive algebraic subgroup of
the isometry group, the study of such actions is reduced to considering the
action of a compact Lie group on a dual compact symmetric space. This method is
especially convenient for studying polar and hyperpolar actions, since we have
proved that an action is (hyper)polar if and only if its dual action is
(hyper)polar. Using this fact we are able to generalize a number of
classification results from the compact to the noncompact setting. In
particular, this provides many new examples of polar and hyperpolar actions on
symmetric spaces of the noncompact type. However, the relation given by duality
between isometric Lie group actions on a symmetric space of the noncompact type
on the one hand and on a compact dual on the other hand is only a partially
defined map, as there are examples of (polar) actions on noncompact symmetric
spaces, e.g.\ homogenous foliations by horospheres on hyperbolic space, for
which no dual action exists. Nevertheless, the method covers an important
aspect of polar actions in the noncompact setting. Indeed, it is an interesting
question if the methods developed by Berndt, D\'{\i}az-Ramos and Tamaru
\cite{BDRT08}, \cite{bt1}, \cite{bt2}, \cite{bt3} can be combined with our
approach to obtain complete classification results for (hyper)polar actions on
symmetric spaces of the noncompact type. It is conceivable that the method
described in this article has further potential applications beyond the theory
of polar actions.


\end{document}